\documentclass[a4paper,twoside]{article}

\usepackage[top=3cm, bottom=3cm, left=3cm, right=3cm]{geometry}

\usepackage[utf8]{inputenc}
\usepackage{bm}
\usepackage{amsfonts}
\usepackage{amsmath,amssymb,amsthm,mathtools}
\usepackage{comment}
\usepackage{enumitem}
\usepackage{cite}
\usepackage{color}
\usepackage{ulem}
\usepackage{tikz}
\newtheorem{theorem}{Theorem}[section]

\newtheorem{lemma}[theorem]{Lemma}
\theoremstyle{definition}
\newtheorem{rem}[theorem]{Remark}
\providecommand{\abs}[1]{\lvert#1\rvert}

\numberwithin{equation}{section}

\title{Mean curvature flows of graphs \\ sliding off to infinity in warped product manifolds}
\author{Naotoshi Fujihara}
\date{}

\begin{document}

\maketitle

\begin{abstract}
    We study mean curvature flows in a warped product manifold defined by a closed Riemannian manifold and $\mathbb{R}$. 
    In such a warped product manifold, we can define the notion of a graph, called a geodesic graph. 
    We prove that the curve shortening flow preserves a geodesic graph for any warping function, and the mean curvature flow of hypersurfaces preserves a geodesic graph for some monotone convex warping functions. 
    In particular, we consider some warping functions that go to zero at infinity, which means that the curves or hypersurfaces go to a point at infinity along the flow. 
    In such a case, we prove the long-time existence of the flow and that the curvature and its higher-order derivatives go to zero along the flow.
\end{abstract}


\section{Introduction}
\label{sec:Introduction}

Let $F \colon M \times [0,T) \to \overline{M}$ be a smooth mapping and $F_t = F(\cdot,t)$ an immersion of an $n$-dimensional closed orientable manifold $M$ into an $(n+1)$-dimensional orientable Riemannian manifold $(\overline{M},\overline{g})$ for every $t \in [0,T)$. 
The family of immersions $\{F_t\}_{t \in [0,T)}$ is called a \textit{mean curvature flow} if it is a solution to the equation 
\begin{equation}
    \frac{\partial F}{\partial t} = - H_t N_t, \label{eq:MCF}
\end{equation}
where $N_t$ is a unit normal vector field of $F_t(M)$ and $H_t$ is a mean curvature of $F_t$ with respect to $-N_t$. 
When $n=1$, the mean curvature flow is especially called a \textit{curve shortening flow}. 
The short-time existence and uniqueness of the solution to (\ref{eq:MCF}) are guaranteed for an immersion of closed manifolds, and the mean curvature flow has been studied extensively. 
In particular, for general ambient spaces, it has been studied by Huisken \cite{huisken1986contracting}, Grayson \cite{grayson1989shortening} and Gage \cite{gage1990curve}, for example.
In this paper, we shall consider a warped product manifold for an ambient space, and we will study the graphical motion of the mean curvature flow in that space. 

We then introduce warped product manifolds. 
Let $(M, g_M)$ be an $n$-dimensional closed Riemannian manifold, and $r(z)$ a positive function defined on an open interval $I$.
A warped product manifold $(\overline{M}, \overline{g})$ of $M$ and $I$ with a warping function $r$ is defined by a product manifold $M \times I$ with a metric 
\begin{align*}
\overline{g} = r(z)^2 g_M + dz \otimes dz.
\end{align*}
The metric $\overline{g}$ is also denoted by $\langle \cdot, \cdot \rangle$.
The warped product manifold includes spaces of constant curvature. 
We give examples when $n = 1$, that is, $(M,g_M)$ is a unit circle $(\mathbb{S}^1,d \theta)$.
\begin{itemize}
\item
We set $I = (0, \infty)$ and $r(z) = z$. The metric
\begin{align*}
\overline{g} = z^2 d \theta \otimes d \theta + d z \otimes d z
\end{align*}
defines a space of constant (Gauss) curvature $0$.
\item
For $k > 0$, we set $I = (-\pi/\sqrt{k}, \pi/ \sqrt{k})$ and $r(z) = \cos \sqrt{k} z$. The metric
\begin{align*}
\overline{g} = \cos^2 (\sqrt{k} z )d \theta \otimes d \theta + d z \otimes d z
\end{align*}
defines a space of constant (Gauss) curvature $k$.
\item
For $k > 0$, we set $I = (-\infty, \infty)$ and $r(z) = e^{\sqrt{k} z}$. The metric
\begin{align*}
\overline{g} = e^{2 \sqrt{k} z} d \theta \otimes d \theta + d z \otimes d z
\end{align*}
defines a space of constant (Gauss) curvature $-k$.
\end{itemize}

We shall define the notion of a graph in $(\overline{M}, \overline{g})$.  
A hypersurface in $(\overline{M}, \overline{g})$ is a geodesic graph (simply called a graph in this paper) if and only if the angle function $\langle \partial_z, N \rangle$ is positive, where $\partial_z$ is a coordinate vector in $\overline{M}$ and $N$ is a unit normal vector to the hypersurface. 
The graphical mean curvature flow of hypersurfaces in Euclidean space was studied by Ecker-Huisken \cite{ecker1989entire,ecker1991interior}. 
For hypersurfaces in warped product manifolds, it is studied by Borisenko and Miquel \cite{borisenko2012}, and by Huang, Zhan, and Zhou \cite{huang2019}.
In \cite{borisenko2012}, they studied the case of warped product manifolds $(M \times I, g_M + \phi(p)^2 dz \otimes dz)$, where $(M, g_M)$ is a Riemannian manifold and $\phi$ is a positive function on $M$. 
In the same way as above, the angle function for hypersurfaces in this warped product manifold can be defined, and if the angle function is positive, the hypersurface is called an equidistant graph. 
In \cite{huang2019}, they consider the geodesic graph for $\dim M \geq 2$. 
The author studied the curve shortening flow on the surfaces of revolution in \cite{F}.
The warped product manifolds can be considered as a generalization of the surface of revolution.
In fact, we can reparameterize the axis of rotation, and then we can obtain the metric of warped products.
Theorem \ref{thm:1} is a generalization of the theorem in \cite{F}. 

\begin{theorem}
\label{thm:1}
Let $(\overline{M}, \overline{g}= r(z)^2 d \theta \otimes d \theta + d z \otimes d z)$ be a warped product manifold of a unit circle $(\mathbb{S}^1, d \theta \otimes d \theta)$ and an interval $(I, dz \otimes dz)$ with a smooth positive function $r \colon I \to \mathbb{R}$. 
We assume $F \colon \mathbb{S}^1 \times [0,T) \to \overline{M}$ to be a curve shortening flow with $F_0$ a graph. 
Then the following statements hold:
\begin{itemize}
\item[(i)] The flow preserves a graph on $[0,T)$;
\item[(ii)] Assume $I = (-\infty, a)$ and $r'(z) > 0$ on $I$. 
\begin{itemize}
\item[(a)] If $\sup_{z \in I} r'(z)/r(z) < \infty$, the flow $\{ F_t \}_t$ exists on $[0,\infty)$. 
\item[(b)] If $r(z)$ satisfies $r(z) r''(z) - 2 {r'(z)}^2 \geq 0$ on $I$, then we have 
\begin{align*}
    \kappa_t \to 0 \, (t \to \infty).
\end{align*}
Furthermore, if we have $\sup_{z \in I}\abs{ r^{(i)}(z) / r(z) } < \infty$ for all $i \geq 2$, then we also have 
\begin{align*}
    \partial_{s}^{(m)} \kappa_t \to 0 \, (t \to \infty)
\end{align*}
for all $m \geq 1$, where $\partial_s$ is a derivative with respect to the arc length $s$.
\end{itemize}
\end{itemize} 
\end{theorem}
We give an example of a function $r(z)$ that satisfies the above assumption.
We set $r(z) = (-z)^{-\beta}$ for $z < 0$ and $\beta > 0$, then we have 
\begin{align*}
\frac{r'(z)}{r(z)} = -\frac{\beta}{z}, \quad r(z) r''(z) - 2 {r'(z)}^2 = \beta  (1 - \beta) (-z)^{-2 (\beta + 1)}.
\end{align*}
Thus, we define $I = (-\infty, a)$ for any $a < 0$, then the function $r \colon I \to \mathbb{R}$ satisfies the first condition; the second condition is satisfied when we set $\beta \leq 1$. 
For $\sup_{z \in I} r'(z)/ r(z) = \infty$, we will make some remarks in Section \ref{sec:case1}.
We then move on to the case $n = \dim M \geq 2$. 
The property that the flow preserves a geodesic graph is not expected to work well in a general setting, and, actually we can construct an example that does not preserve its graph property along the mean curvature flow (see the appendix).
In \cite{huang2019}, Huang, Zhan, and Zhou assume the following four conditions for $r \colon (-a, a) \to (0,\infty)$ and $(M, g_M)$:
\begin{itemize}
\item[(0)] $\mathrm{Ric}_M \geq n \rho g_M$; 
\item[(1)] $r(0) = 1$ and  $r'(0) = 0$; 
\item[(2)] $r'(z) > 0$ for all  $z \in (0,a)$ and $r'(z) < 0$ for all  $z \in (-a,0)$;
\item[(3)] $r(z) r''(z) - {r'(z)}^2 + \rho \geq c$,
\end{itemize}
where $\rho \in \mathbb{R}$ and $c = \max \{ 0, \rho \}$. 
Under these conditions, they showed that the flow with some initial condition remains a graph, and it exists for all time and converges to $M \times \{0\}$. 
We then consider a similar setting; Let $r$ be defined on $I = (-\infty, a)$ and let $r'(z) > 0$ hold for all $z \in I$. 
This can be considered as a right half of the above setting.
Therefore, if $\lim_{z \to -\infty} r(z) = 1$ (or converges to any other positive constant), the flow remains a geodesic graph under conditions (0) and (3).
However, if $\lim_{z \to -\infty} r(z) = 0$, the proof in \cite{huang2019} cannot be applied, since $r(0) > 0$ is essentially used in its proof.
Thus, we replace condition (3) by 
\begin{align*}
r(z) {r''}(z) - (1 + \alpha) {r'}(z)^2 + \rho \geq c,
\end{align*}
where  $c = \max\{\rho, 0\}$ and $\alpha > 1$. 
This condition is considered as a generalization of the inequality $r r'' - 2 {r'}^2 \geq 0$. 
In this paper, we will prove the following theorem.
\begin{theorem}
\label{thm:2}
Let $I = (-\infty, a)$ and $r \colon I \to \mathbb{R}$ be a positive smooth function and set $(\overline{M},\overline{g}) = (M \times I, r(z)^2 g_M + d z \otimes d z)$, where $(M, g_M)$ is a closed Riemannian manifold. 
Assume the following three conditions:
\begin{itemize}
\item[$(C0)$] $\mathrm{Ric}_M \geq n \rho g_M$; 
\item[$(C1)$] $r'(z) > 0$ for all  $z \in (-\infty,a)$; and
\item[$(C2)$] $r(z) r''(z) - (1 + \alpha){r'(z)}^2 + \rho \geq c$,
\end{itemize}
where $\rho$ is a real number, $c = \max \{ 0, \rho \}$, and $\alpha > 1$. 
If $\{ F_t \}_{t\in[0,T)}$ is a mean curvature flow and satisfies the initial condition:
  \begin{align*}
  \min_{p \in M_0} \Theta_0(p) > \alpha^{-1/2}, 
  \end{align*}
  where $\Theta_0$ is the angle function defined by $\Theta_t = \langle \partial_z, N_t \rangle$ along the flow, then the following statements hold:
  \begin{itemize}
  \item[(i)] the flow preserves a graph and exists on $[0,\infty)$;
  \item[(ii)] if $R_M = 0$ and $r^{(i)}(z)/ r(z) \to 0$ as $z \to -\infty$ for $i \geq 1$, we have
 \begin{align*}
 \abs{\nabla^{m} A} \to 0 \, (t \to \infty)
 \end{align*}
 for all $m \geq 0$.
  \end{itemize}
\end{theorem}

This paper is organized as follows. 
We first compute some basic geometric quantities about warped product manifolds and show some properties related to the warping function and the Riemann curvature tensor in Section \ref{sec:basic}. 
In Section \ref{sec:evolution}, we calculate the evolution equations that play an important role in the proof of Theorem \ref{thm:1} and Theorem \ref{thm:2}. Theorem \ref{thm:1} is shown in Section \ref{sec:case1}. We will also give some remarks about the examples that do not satisfy the assumption of the theorem. In Section \ref{sec:case2}, we will prove Theorem \ref{thm:2}. 
In the appendix, we will give an example that does not remain a graph along the mean curvature flow in a warped product manifold.

\section{Basic formulas of warped product manifolds}
\label{sec:basic}

In this section, we give some basic geometric quantities of warped products.

First, we compute the covariant derivative with respect to $\overline{g}$.
To compute the covariant derivatives, we will use a local orthonormal frame.
Let $\{ E^M_i \}$ be a local orthonormal frame on $(M, g_M)$ and we define
\begin{align*}
    E_i=\frac{1}{r(z)} E^M_i \quad E_z= \partial_z,
\end{align*}
and $\theta^i$, $\theta^z$ denote the dual basis, then we have
$\theta^i = r \, \theta_{M}^i$, $\omega^z = d z$, where $\{ \theta_M^i \}$ is the dual basis of $\{ E^M_i \}$.

\begin{lemma}
\label{lem:nabla bar 1}
For $n=1$, we have 
\begin{align*}
    &\overline{\nabla}_{E_{\theta}} E_{\theta} = - \frac{r'}{r}  E_z, \quad \overline{\nabla}_{E_z} E_{\theta} = 0, \\
    &\overline{\nabla}_{E_{\theta}} E_z = \frac{r'}{r} \, E_{\theta}, \quad \overline{\nabla}_{E_z} E_z = 0. \\ 
\end{align*}
and Gauss curvature
\begin{align*}
\overline{K} &= \overline{\mathrm{Ric}}(\overline{v},\overline{v})=-\frac{r''}{r},
\end{align*}
where $\overline{v}$ is the unit tangent vector of $(\overline{M},\overline{g})$. The Gauss curvature $\overline{K}$ does not depend on $\theta$, so we sometimes write $\overline{K}(z)$. 
\end{lemma}
\begin{lemma}
\label{lem:nabla bar 2}
For $n \geq 2$, we have
\begin{align*}
    &\overline{\nabla}_{E_i} E_j = \frac{1}{r^2} \nabla^{M}_{E^M_i} E^M_j - \frac{r'}{r} \delta_{i j}E_z, \quad \overline{\nabla}_{E_z} E_i = 0, \\
    &\overline{\nabla}_{E_i} E_z = \frac{r'}{r} E_i, \quad \overline{\nabla}_{E_z} E_z = 0, \\ 
\end{align*}
where $\nabla^{M}$ is a Levi-Civita connection of $(M, g_M)$, and we have Riemann curvature tensor $\overline{R}$
\begin{align*}
\overline{R}_{i j k l} &=  r^{-2} (R_M)_{i j k l} - \frac{{r'}^2}{r^2} \delta_{i k} \delta_{j l},\\
\overline{R}_{i z \alpha \beta} &= - \frac{r''}{r} \delta_{i \alpha} \delta_{z \beta},
\end{align*}
where $i < j$, $k <  l$ for $i, j ,k ,l = 1, \dots, n$ and $\alpha < \beta$ for $\alpha, \beta=1, \dots, n, z$. The others are zeros.
\end{lemma}

\begin{proof}
We compute Cartan connection forms $\omega^{i}_{j}$, $\omega^{i}_{z}$ that satisfy
\begin{align*}
    d \theta^{i} = - \sum_{j=1}^{n} \omega^{i}{}_{j} \wedge \theta^j - \omega^{i}{}_{z} \wedge \theta^z, \quad
    d \theta^{z} = - \sum_{j=1}^{n} \omega^{z}{}_{j} \wedge \theta^{j}, \quad \omega^{i}{}_{z}=-\omega^{z}{}_{i}.
\end{align*}
The exterior derivatives of $\theta^{i}$, $\theta^z$ are
\begin{align*}
    d \theta^i 
    &= r' d z \wedge \theta_{M}^i + r d \theta_{M}^i \\
    &= -\frac{r'}{r} \theta^i \wedge \theta^z - \sum_{j = 1}^{n} {(\omega_{M})}^{i}{}_{j} \wedge \theta^{j}, \\
    d \theta^z &= d d z = 0, 
\end{align*}
where ${(\omega_M)}^{i}{}_{j}$ are Cartan connection forms of $(M, g_M)$. Thus we obtain 
\begin{align*}
\omega^{i}{}_{j} = {(\omega_{M})}^{i}{}_{j}, \quad \omega^{i}{}_{z} = \frac{r'}{r} \theta^i = r' \theta_{M}^{i}.
\end{align*}
 By using $\overline{\nabla}_{X} E_i = \omega^{k}{}_{i}(X) \, E_k + \omega^{z}{}_{i}(X) E_z$,  we can obtain the first formulas. 
Then we compute the curvature form $\Omega$.
\begin{align*}
\Omega^{i}{}_{j} 
&= d \omega^{i}{}_{z} + \sum \omega^{i}{}_{k} \wedge \omega^{k}{}_{z} \\
&= d {(\omega_M)}^{i}{}_{j} + \sum {(\omega_M)}^{i}{}_{k} \wedge {(\omega_M)}^{k}{}_{j} + \frac{r'}{r} \theta^{i} \wedge -\frac{r'}{r} \theta^{j} \\
&= {(\Omega_M)}^{i}{}_{j} - \frac{{r'}^2}{r^2} \theta^i \wedge \theta^{j},\\
\Omega^{i}{}_{z} 
&= d \omega^{i}{}_{z} + \sum \omega^{i}{}_{k} \wedge \omega^{k}{}_{j} \\
&=r'' \theta^z \wedge \theta_{M}^{i} + r' d \theta_{M}^{i} + \sum {(\omega_M)}^{i}{}_{k} \wedge r' \theta_{M}^{k} \\
&= - \frac{r''}{r} \theta^i \wedge \theta^z.
\end{align*}
From $\Omega^{\gamma}{}_{\delta} = \sum_{\alpha < \beta} \overline{R}^{\gamma}{}_{\delta \alpha \beta} \theta^{\alpha} \wedge \theta^{\beta}$, ${(\Omega_M)}^{i}{}_{j} = \sum_{k < l} {(R_M)}^{i}{}_{j k l} \theta_{M}^{k} \wedge \theta_{M}^{l}$, and $\overline{R}^{\gamma}{}_{\delta \alpha \beta} = - \overline{R}^{\gamma}{}_{\delta \beta \alpha}$, we obtain
\begin{align*}
\overline{R}^{i}{}_{j k l} &= r^{-2} (R_M)^{i}{}_{j k l} - \frac{{r'}^2}{r^2} \delta^{i}{}_{k} \delta_{j l},\\
\overline{R}^{i}{}_{z \alpha \beta} &= - \frac{r''}{r} \delta^{i}{}_{\alpha} \delta_{z \beta},
\end{align*}
for $i < j$. We have $\overline{R}^{i}{}_{j \alpha \beta} = - \overline{R}^{j}{}_{i \alpha \beta}$ and the other components are zeros.
\end{proof}

\begin{lemma}
For $n \geq 2$, the Ricci curvature $\overline{\mathrm{Ric}}$ is as follows:
\begin{align*}
\overline{\mathrm{Ric}} &= \mathrm{Ric}_M - (r r'' +(n-1) {r'}^2) g_M - n \frac{r''}{r} dz \otimes dz.
\end{align*}
\end{lemma}

\begin{lemma}
\label{lem:rRK}
Consider the case $n = 1$ and assume $\abs{r' / r} < \infty$. 
Then the following statements are equivalent:
\begin{itemize}
\item[$(1)$] $\abs{r^{(i)} / r} < \infty$ for all $i \geq 2$; 
\item[$(2)$] $\abs{\overline{\nabla}^{i} \overline{R}} < \infty$ for all $i \geq 0$;
\item[$(3)$] $\abs{\overline{K}^{(i)}} < \infty$ for all $i \geq 0$.
\end{itemize}
\end{lemma}

\begin{proof}
From Lemma \ref{lem:nabla bar 1}, we have
\begin{align*}
\overline{R}(E_{\theta}, E_{z}, E_{z}, E_{\theta}) = \overline{K} = - \frac{r''}{r},
\end{align*}
then we can compute all the derivatives, for example;
\begin{align*}
( \overline{\nabla}_{E_z} \overline{R}) (E_{\theta}, E_{z}, E_{z}, E_{\theta}) = \overline{K}' = - \frac{r'''}{r} + \frac{r''}{r} \frac{r'}{r}.
\end{align*}
This proves the lemma.
\end{proof}

\begin{lemma}
\label{lem:rR}
Let $r \colon (-\infty, a) \to (0,\infty)$ and $n \geq 2$.
If $R_M = 0$, the following two statements are equivalent:
\begin{itemize}
\item[$(1)$] $\abs{r^{(i)} / r} \to 0$ as $z \to -\infty$ for all $i \geq 1$; 
\item[$(2)$] $\abs{\overline{\nabla}^{i} \overline{R}} \to 0$ as $z \to -\infty$ for all $i \geq 0$.
\end{itemize}
\end{lemma}
\begin{proof}
Since $R_M = 0$, we can take $\{E^M_i\}$ that satisfies $\nabla^{M}_{E^M_i} E^M_j = 0$.
From Lemma \ref{lem:nabla bar 2}, we have
\begin{align*}
\overline{R}_{i j i j } &=  - \frac{{r'}^2}{r^2} \delta_{i i} \delta_{j j} = - \frac{{r'}^2}{r^2},\\
\overline{R}_{i z i z} &= - \frac{r''}{r} \delta_{i i} \delta_{z z} = - \frac{r''}{r} ,
\end{align*}
thus as in the proof of Lemma \ref{lem:rRK}, we have for the higher-order derivatives of curvature tensor,
\begin{align*}
\abs{\overline{\nabla}^{i} \overline{R}}^2 = P \left(\frac{r'}{r}, \frac{r''}{r}, \dots, \frac{r^{(i+2)}}{r} \right),
\end{align*}
where P is a polynomial that has no constant term.
Using this, we can prove the lemma.
\end{proof}

\begin{lemma}
\label{lem:curv-dist}
If $I=(-\infty, a)$, $r'(z) > 0$ and $r(z) r''(z) - (1 + \alpha) {r'}(z)^2 \geq 0$ for $\alpha > 0$, then we have
\begin{align*}
\frac{r'}{r}(z_2) - \frac{r'}{r}(z_1) \leq - \alpha (z_1 - z_2) \frac{r'}{r}(z_1) \frac{r'}{r}(z_2),
\end{align*}
for $-\infty < z_2 < z_1 < a$. 
In particular, $r'/r$ is monotone and goes to zero as $z \to -\infty$.
\end{lemma}
\begin{proof}
Set $w(z) = r'(z)/r(z)$ and we have 
\begin{align*}
\frac{d w}{d z}(z) = \frac{r(z) r''(z) - {r'(z)}^2}{r(z)^2} \geq \alpha \frac{{r'(z)}^2}{r(z)^2} =  \alpha w(z)^2,
\end{align*}
thus $w(z)$ is monotone. Then we divide it by $- w(z)^2$ and integrate it from $z_2$ to $z_1$ $(z_2 < z_1)$ to obtain
\begin{align*}
\frac{1}{w(z_1)} - \frac{1}{w(z_2)} \leq  -\alpha(z_1 - z_2).
\end{align*}
The right-hand side goes to $-\infty$ as $z_2 \to -\infty$, so $w(z)$ goes to zero.
\end{proof}

\section{The evolution equations}
\label{sec:evolution}

We shall calculate the evolution equations in this section.
Let $F_t \colon M \to \overline{M}$, $t \in [0,T)$ be a smooth family of immersions that satisfies equation (\ref{eq:MCF}).
The Levi-Civita connection induced in $M$ by $F_t$ is denoted by $\nabla$
and $\overline{\nabla}^F$ denotes the connection of $\overline{\nabla}$ induced by $F$.
We then define an angle function on $M$ as already mentioned in Section \ref{sec:Introduction},
\begin{align}
\label{eq:functions}
   \Theta_t = \langle N_t, E_z \rangle,
\end{align}
and we define a height function $z_t \colon M \to I$ by
\begin{align*}
z_t = \pi_z \circ F_t,
\end{align*}
where the projection from $\overline{M}$ to $I$ is denoted by $\pi_z$.
First, we compute the evolution equation for the angle function. 
Let us define $v_M$ by
\begin{align*}
v_M = \frac{N_M}{\abs{N_M}}_M, \quad N_M = \sum_{k=1}^{n} \langle N, E_k \rangle E_k,
\end{align*}
if $\abs{N_M} = 0$, we define $v_M = 0$. The vector $N_M$ is the tangential component of the unit normal vector $N$ to $M$.
\begin{lemma}
\label{lem:theta}
The evolution equations for $\Theta$ are as follows:
for $n=1$,
\begin{align*}
    \left(\partial_t - \Delta\right) \Theta= \frac{r r'' - 2 {r'}^2}{r^2} \Theta (1 - \Theta^2) + \left( \frac{r'}{r} \Theta - \kappa \right)^2 \Theta,
\end{align*}
for $n \geq 2$,
\begin{align*}
    \left(\partial_t - \Delta\right) \Theta
    &= 2 \frac{r'}{r}\{ \langle \nabla \Theta, E_z \rangle - H \} + \abs{A}^2 \Theta \\
    & \quad + n \frac{{r'}^2}{r^2} \Theta + \frac{\Theta (1 - \Theta^2)}{r^2} \{ n (r r'' - {r'}^2) + \mathrm{Ric}_M(v_M, v_M) \},
\end{align*}
\end{lemma}

\begin{proof}
For $n\geq 2$, see the proofs of Theorem 3.1 and related propositions in \cite{huang2019}. 
We only compute the case $n=1$. 
We have for the arc length parameter $s$, $\partial_s = \partial_x / \abs{\partial_x}$ and $\mathfrak{t} = d F(\partial_s)$, where $x$ is a coordinate of $I$. 
Using the frame $\{ E_{\theta}, E_z \}$ we can express $\mathfrak{t}$, $N$ as follows:
\begin{align*}
\mathfrak{t} &= \langle N, E_z \rangle E_{\theta} - \langle N, E_{\theta} \rangle E_z, \\
N &= \langle N, E_{\theta} \rangle E_{\theta} + \langle N, E_z \rangle E_z.
\end{align*}
We then have
\begin{align*}
\partial_s \Theta 
&= \langle \overline{\nabla}^{F}_{\partial_s} E_z, N \rangle + \langle E_z, \overline{\nabla}^{F}_{\partial_s} N \rangle \\
&= \left\langle \frac{r'}{r} \Theta E_{\theta}, N \right\rangle + \langle E_z, \kappa \mathfrak{t} \rangle \\
&= \left( \frac{r'}{r} \Theta - \kappa \right) \langle N, E_{\theta} \rangle,
\end{align*}
and we also have 
\begin{align*}
\partial_s \langle N, E_{\theta} \rangle = - \left( \frac{r'}{r} \Theta - \kappa \right) \Theta.
\end{align*}
Hence we obtain 
\begin{align*}
\Delta \Theta 
&= \partial_s \partial_s \Theta \\
&= \partial_s \left\{ \left( \frac{r'}{r} \Theta - \kappa \right) \langle N, E_{\theta} \rangle \right\} \\
&= \partial_s \left( \frac{r'}{r} \Theta - \kappa \right)  \langle N, E_{\theta} \rangle + \left( \frac{r'}{r} \Theta - \kappa \right) \partial_s \langle N, E_{\theta} \rangle \\
&= - \left( \frac{r'}{r} \right)' \Theta (1- \Theta^2) + \frac{r'}{r} \left( \frac{r'}{r} \Theta - \kappa \right)(1 - \Theta^2)
- \partial_s \kappa \langle N, E_{\theta} \rangle - \left( \frac{r'}{r} \Theta - \kappa \right)^2 \Theta,
\end{align*}
where we use $\partial_s z = d \pi_z (\mathfrak{t}) = - \langle N, E_{\theta} \rangle$ and $\langle N, E_{\theta} \rangle^2 = 1 - \Theta^2$.
The time derivative of $\Theta$ is as follows:
\begin{align*}
\partial_t \Theta 
&= \langle \overline{\nabla}^{F}_{\partial_t} E_z, N \rangle + \langle E_z, \overline{\nabla}^{F}_{\partial_t} N \rangle \\
&=  \langle \overline{\nabla}^{F}_{- \kappa N} E_z, N \rangle + \langle E_z, d F (\mathrm{grad}\kappa) \rangle \\
&= - \kappa \langle N, E_{\theta} \rangle \langle \overline{\nabla}_{E_{\theta}} E_z, N \rangle + \langle E_z, \partial_s \kappa \mathfrak{t} \rangle \\
&= - \kappa \langle N, E_{\theta} \rangle \left\langle \frac{r'}{r} E_{\theta}, N \right\rangle + \partial_s \kappa \langle E_z, N \rangle \\
&= - \kappa \frac{r'}{r} (1 - \Theta^2) - \partial_s \kappa \langle N, E_{\theta} \rangle.
\end{align*}
Thus we obtain 
\begin{align*}
(\partial_t - \Delta) \Theta
&= - \kappa \frac{r'}{r} (1 - \Theta^2) - \partial_s \kappa \langle N, E_{\theta} \rangle \\
 &+ \partial_s \kappa \langle N, E_{\theta} \rangle + \frac{r r'' - {r'}^2}{r^2} \Theta (1 - \Theta^2) 
 - \frac{r'}{r} \left( \frac{r'}{r} \Theta - \kappa \right)(1-\Theta^2) + \left( \frac{r'}{r} \Theta - \kappa \right)^2 \Theta \\
& = \frac{r r'' - 2 {r'}^2}{r^2} \Theta (1 - \Theta^2) +  \left( \frac{r'}{r} \Theta - \kappa \right)^2 \Theta.
\end{align*}
\end{proof}

\begin{lemma}
\label{lem:v}
Let $v = \Theta^{-1}$, then we have
for $n=1$,
\begin{align*}
    \left(\partial_t - \Delta\right) v= -\frac{2}{v} \abs{\partial_s v}^2 -\frac{r r'' - 2 {r'}^2}{r^2}  \left( v - \frac{1}{v} \right) - \left( \frac{r'}{r} \Theta - \kappa \right)^2 v,
\end{align*}
for $n \geq 2$,
\begin{align*}
    \left(\partial_t - \Delta\right) v
    &= -\frac{2}{v} \abs{\nabla v}^2 +2 \frac{r'}{r} \langle \nabla v, E_z \rangle +2 \frac{r'}{r} H v^2 - \abs{A}^2 v \\
    & \quad - n \frac{{r'}^2}{r^2} v - \frac{1}{r^2} \{ n (r r'' - {r'}^2) + \mathrm{Ric}_M(v_M, v_M) \} \left( v - \frac{1}{v} \right).
\end{align*}
\end{lemma}

The evolution equations for the curvatures are given as in \cite{huisken1986contracting}.
\begin{lemma}
\label{lem:curvatures}
We have the evolution equations for the curvature;
for $n = 1$,
\begin{align*}
(\partial_t - \Delta) \kappa^2 &= -2 \abs{\partial_s \kappa}^2 + 2 \kappa^2 \left( \kappa^2 - \frac{r''}{r} \right), 
\end{align*}
for $n \geq 2$,
\begin{align*}
(\partial_t - \Delta) \abs{A}^2
&= -2 \abs{\nabla A}^2 + 2 \abs{A}^2 \left( \abs{A}^2 + \overline{\mathrm{Ric}}(N,N) \right) \\
&-4 \left( h^{i j} h^{m}{}_{j} \overline{R}_{m l i}{}^{l} - h^{i j} h^{l m} \overline{R}_{m i l j} \right)
- 2 h^{i j} \left( \overline{\nabla}_i \overline{R}_{0 l j}{}^{l} + \overline{\nabla}_l \overline{R}_{0 j i}{}^{l} \right),
\end{align*}
where $\overline{\mathrm{Ric}}(N,N) = \overline{R}_{0 i 0}{}^{i}$.
\end{lemma}

We then derive the inequality for the following function $\mathfrak{g}$ in order to prove the long-time existence of the flow. 
This function was originally used in \cite{ecker1991interior}.
\begin{lemma}
\label{lem:g}
We assume that $\sup_{(p,t) \in M \times [0, t_0)} v_t(p) < \infty$ for some $t_0 \in (0,T]$ and define $\mathfrak{g}$ as follows;
\begin{align*}
\mathfrak{g}_t = \varphi(v_t) \abs{A_t}^2, \quad \varphi(v) = \frac{v^2}{1 - k v^2}, \quad k = \frac{1}{2 \sup_{(p,t) \in M \times \in [0, t_0)} {v_t}(p)^2}.
\end{align*}
We have the following inequalities that hold in $M \times [0,t_0)$:
for $n=1$,
\begin{align*}
(\partial_t - \Delta) \mathfrak{g}
\leq - 2 k \mathfrak{g}^2 + \frac{4 r'}{r} \frac{\sqrt{\varphi}}{v^3} \mathfrak{g}^{3/2} 
- \left[ \frac{r r'' - 2 {r'}^2 }{r^2} \left(v - \frac{1}{v}\right) \frac{2 \varphi}{v^3}  + 2 \frac{r''}{r} \right] \mathfrak{g},
\end{align*}
for $n \geq 2$,
\begin{align*}
    \left(\partial_t - \Delta\right) \mathfrak{g}
    &\leq - 2 k \mathfrak{g}^2 + 4 \sqrt{n} \frac{r'}{r} \frac{\sqrt{\varphi}}{v} \mathfrak{g}^{3/2} + 4 \sqrt{\varphi} \abs{\overline{\nabla} \overline{R}} \sqrt{\mathfrak{g}} \\
   & + \left[ - \frac{2}{r^2} \frac{\varphi}{v^3} \{ n (r r'' - {r'}^2) + \mathrm{Ric}_M(v_M,v_M) \} \left( v - \frac{1}{v}\right) + 2 \overline{\mathrm{Ric}}(N,N) + 8 \abs{\overline{R}}  \right] \mathfrak{g}.
\end{align*}
\end{lemma}

\begin{proof}
When $n=1$ we can write $\mathfrak{g} = \varphi(v) \kappa^2$, we have
\begin{align*}
(\partial_t - \Delta) \mathfrak{g} = \kappa^2 \varphi' (\partial_t - \Delta) v + \varphi(v) (\partial_t - \Delta) \kappa^2 - \varphi'' \abs{\partial_s v}^2 \kappa^2 - 2 ( \partial_s \varphi \cdot \partial_s \kappa^2 ).
\end{align*}
We then evaluate $\partial_s \varphi \cdot \partial_s \kappa^2$. We have
\begin{align*}
- \partial_s \varphi \cdot \partial_s \kappa^2  = - \frac{1}{\varphi} ( \partial_s \mathfrak{g} \cdot \partial_s \varphi ) + \kappa^2 \frac{\abs{\partial_s \varphi}^2}{\varphi},
\end{align*}
and 
\begin{align*}
- \partial_s \varphi \cdot  \partial_s \kappa^2 
&=  - \kappa \partial_s \varphi \cdot  2 \partial_s \kappa \\
&\leq \frac{\abs{\partial_s \varphi}^2}{2 \varphi} \kappa^2 + 2 \varphi \abs{\partial_s \kappa}^2,
\end{align*}
where we use Young's inequality $a b \leq a^2 / 2 \varepsilon + \varepsilon b^2 / 2$ with $\varepsilon = \varphi$. 
From the above, we have
\begin{align*}
-2 ( \partial_s \varphi \cdot \partial_s \kappa^2 )
\leq - \frac{1}{\varphi} ( \partial_s \mathfrak{g} \cdot \partial_s \varphi ) + 2 \varphi \abs{\partial_s \kappa}^2 + \frac{3}{2 \varphi} \abs{\partial_s \varphi}^2 \kappa^2.
\end{align*}
From Lemma \ref{lem:curvatures}, we obtain
\begin{align*}
(\partial_t - \Delta) \mathfrak{g}
&\leq \kappa^2 \varphi' \left[ -\frac{2}{v} \abs{\partial_s v}^2 -\frac{r r'' - 2 {r'}^2}{r^2}  \left(v - \frac{1}{v} \right) - \left( \frac{r'}{r} \Theta - \kappa \right)^2 v \right] \\
&+ \varphi(v) \left[ -2 \abs{\partial_s \kappa}^2 + 2 \kappa^2 \left( \kappa^2 - \frac{r''}{r} \right) \right] \\
&- \varphi'' \abs{\partial_s v}^2 \kappa^2 - \frac{1}{\varphi} ( \partial_s \mathfrak{g} \cdot \partial_s \varphi )
+ 2 \varphi \abs{\partial_s \kappa}^2 + \frac{3}{2 \varphi} \abs{\partial_s \varphi}^2 \kappa^2 \\
&= \kappa^2 \varphi' \left[ -\frac{2}{v} \abs{\partial_s v}^2 -\frac{r r'' - 2 {r'}^2}{r^2}  \left( v - \frac{1}{v} \right) - \frac{1}{v}\frac{{r'}^2}{r^2} + \frac{2 r'}{r} \kappa - \kappa^2 v \right] \\
&+ 2 \mathfrak{g} \left( \kappa^2 - \frac{r''}{r} \right) - \varphi'' \abs{\partial_s v}^2 \kappa^2 - \frac{1}{\varphi} ( \partial_s \mathfrak{g} \cdot \partial_s \varphi ) + \frac{3}{2 \varphi} \abs{\partial_s \varphi}^2 \kappa^2.
\end{align*}
We have the following identities
\begin{align*}
\varphi' &= \frac{2 \varphi^2}{v^3}, \quad \varphi'' = \frac{8 \varphi^3 - 6 \varphi^2 v^2}{v^6},
\end{align*}
and using them, we obtain
\begin{align*}
&- \frac{2}{v} \kappa^2 \varphi' \abs{\partial_s v}^2 - \varphi'' \abs{\partial_s v}^2 \kappa^2 + \frac{3}{2 \varphi} \abs{\partial_s \varphi}^2 \kappa^2 \\
&= - \kappa^2 \abs{\partial_s \varphi}^2 \left( \frac{2}{\varphi' v} + \frac{\varphi''}{{\varphi'}^2} - \frac{3}{2 \varphi} \right) \\
&= - \kappa^2 \abs{\partial_s \varphi}^2 \frac{\varphi - v^2}{2 \varphi^2} \\
&\leq 0,
\end{align*}
where we have $\varphi(v) - v^2 \geq 0$. 
We also have
\begin{align*}
- \kappa^4 \varphi' v + 2 \kappa^2 \mathfrak{g}
&= -2 \frac{\varphi^2 \kappa^4}{v^2} + 2 \frac{\mathfrak{g} \varphi \kappa^2}{\varphi} \\
&= -2 \mathfrak{g}^2 \left( \frac{1}{v^2} - \frac{1 - k v^2}{v^2} \right) \\
&= -2 k \mathfrak{g}^2,
\end{align*}
and 
\begin{align*}
&- \frac{1}{v} \frac{{r'}^2}{r^2} \kappa^2 \varphi' + 2 \frac{r'}{r} \kappa^3 \varphi' - \frac{r r'' - 2 {r'}^2 }{r^2} \left(v - \frac{1}{v} \right) \kappa^2 \varphi' - 2 \varphi  \kappa^2 \frac{r''}{r} \\
&= - \kappa^2 \frac{2 \varphi^2}{v^4} \frac{{r'}^2}{r^2} + \frac{2 r'}{r} \kappa^3 \frac{2 \varphi^2}{v^3} - \frac{r r'' - 2 {r'}^2 }{r^2} \left( v - \frac{1}{v} \right) \kappa^2 \frac{2 \varphi^2}{v^3} - 2 \varphi \kappa^2 \frac{r''}{r} \\
& \leq \frac{4 r'}{r} \frac{\sqrt{\varphi}}{v^3} \mathfrak{g}^{3/2} - \frac{r r'' - 2 {r'}^2 }{r^2} \left( v - \frac{1}{v} \right) \frac{2 \varphi}{v^3} \mathfrak{g} - 2 \frac{r''}{r} \mathfrak{g}.
\end{align*}
Thus, we finally obtain
\begin{align*}
(\partial_t - \Delta) \mathfrak{g}
\leq - 2 k \mathfrak{g}^2 + \frac{4 r'}{r} \frac{\sqrt{\varphi}}{v^3} \mathfrak{g}^{3/2} 
-\left[ \frac{r r'' - 2 {r'}^2 }{r^2} \left( v - \frac{1}{v} \right) \frac{2 \varphi}{v^3}  + 2 \frac{r''}{r} \right] \mathfrak{g}.
\end{align*}

We then compute the case $n \geq 2$ with the same argument as in the case $n=1$.
\begin{align*}
(\partial_t - \Delta) \mathfrak{g} 
&= \abs{A}^2 \varphi' (\partial_t - \Delta) v + \varphi(v) (\partial_t - \Delta) \abs{A}^2 - \varphi'' \abs{\nabla v}^2 \abs{A}^2 - 2 \langle \nabla \varphi, \nabla \abs{A}^2 \rangle \\
&\leq 
\abs{A}^2 \varphi' \biggl[   -\frac{2}{v} \abs{\nabla v}^2 +2 \frac{r'}{r} \langle \nabla v, E_z \rangle +2 \frac{r'}{r} H v^2 - \abs{A}^2 v  \\
 &\qquad\qquad - n \frac{{r'}^2}{r^2} v - \frac{1}{r^2} \{ n (r r'' - {r'}^2) + \mathrm{Ric}_M(v_M, v_M) \} \left(v - \frac{1}{v}\right) \biggr] \\
 &+ \varphi(v) \left[  -2 \abs{\nabla A}^2 + 2 \abs{A}^4 +2 \abs{A}^2 \overline{\mathrm{Ric}}(N,N) 
+8 \abs{A}^2 \abs{\overline{R}}_{\overline{M}} + 4 \abs{A} \abs{\overline{\nabla} \overline{R}}_{\overline{M}} \right] \\
 &- \varphi'' \abs{\nabla v}^2 \abs{A}^2  - \frac{1}{\varphi} \langle \nabla \mathfrak{g}, \nabla \varphi \rangle + 2 \varphi \abs{\nabla A}^2 + \frac{3}{2 \varphi} \abs{\nabla \varphi}^2 \abs{A}^2.
\end{align*}
For $ \langle \nabla \varphi, \nabla \abs{A}^2 \rangle$, we have
\begin{align*}
- \langle \nabla \varphi, \nabla \abs{A}^2 \rangle = - \frac{1}{\varphi} \langle \nabla \mathfrak{g}, \nabla \varphi \rangle + \abs{A}^2 \frac{\abs{\nabla \varphi}^2}{\varphi},
\end{align*}
and also have
\begin{align*}
- \langle \nabla \varphi, \nabla \abs{A}^2 \rangle 
&\leq \frac{\abs{\nabla \varphi}^2}{2 \varphi} \abs{A}^2 + 2 \varphi \abs{\nabla A}^2,
\end{align*}
where we use Young's inequality $a b \leq a^2 / 2 \varepsilon + \varepsilon b^2 / 2$ with $\varepsilon = \varphi$ and Kato's inequality $\abs{\nabla \abs{A}} \leq \abs{\nabla A}$. 
From the above two inequalities, we get
\begin{align*}
-2 \langle \nabla \varphi, \nabla \abs{A}^2 \rangle 
\leq - \frac{1}{\varphi} \langle \nabla \mathfrak{g}, \nabla \varphi \rangle + 2 \varphi \abs{\nabla A}^2 + \frac{3}{2 \varphi} \abs{\nabla \varphi}^2 \abs{A}^2.
\end{align*}
From Lemma \ref{lem:curvatures}, we have
\begin{align*}
(\partial_t - \Delta) \abs{A}^2
&= -2 \abs{\nabla A}^2 + 2 \abs{A}^2 \left( \abs{A}^2 + \overline{\mathrm{Ric}}(N,N) \right) \\
&-4 \left( h^{i j} h^{m}{}_{j} \overline{R}_{m l i}{}^{l} - h^{i j} h^{l m} \overline{R}_{m i l j} \right)
- 2 h^{i j} \left( \overline{\nabla}_j \overline{R}_{0 l i}{}^{l} + \overline{\nabla}_l \overline{R}_{0 i j}{}^{l} \right) \\
&\leq -2 \abs{\nabla A}^2 + 2 \abs{A}^4 +2 \abs{A}^2 \overline{\mathrm{Ric}}(N,N) 
+8 \abs{A}^2 \abs{\overline{R}}_{\overline{M}} + 4 \abs{A} \abs{\overline{\nabla} \overline{R}}_{\overline{M}}.
\end{align*}
Using Young's inequality with $\varepsilon > 0$ (we define it later), we have
\begin{align*}
2 \frac{r'}{r} \langle \nabla v, E_z \rangle \leq \varepsilon \abs{\nabla v}^2 + \frac{{r'}^2}{\varepsilon r^2},
\end{align*}
hence we obtain
\begin{align*}
(\partial_t - \Delta) \mathfrak{g} 
&= \abs{A}^2 \varphi' (\partial_t - \Delta) v + \varphi(v) (\partial_t - \Delta) \abs{A}^2 - \varphi'' \abs{\nabla v}^2 \abs{A}^2 - 2 \langle \nabla \varphi, \nabla \abs{A}^2 \rangle \\
 &\leq 
 \abs{A}^2 \varphi' \biggl[   -\frac{2}{v} \abs{\nabla v}^2 + \varepsilon \abs{\nabla v}^2 + \frac{1}{\varepsilon} \frac{{r'}^2}{r^2} +2 \frac{r'}{r} H v^2 - \abs{A}^2 v  \\
 &\qquad\qquad - n \frac{{r'}^2}{r^2} v - \frac{1}{r^2} \{ n (r r'' - {r'}^2) + \mathrm{Ric}_M(v_M, v_M) \}\left(v - \frac{1}{v} \right) \biggr] \\
 &+ \varphi(v) \left[  2 \abs{A}^4 +2 \abs{A}^2 \overline{\mathrm{Ric}}(N,N) 
+8 \abs{A}^2 \abs{\overline{R}}_{\overline{M}} + 4 \abs{A} \abs{\overline{\nabla} \overline{R}}_{\overline{M}} \right] \\
 &- \varphi'' \abs{\nabla v}^2 \abs{A}^2  - \frac{1}{\varphi} \langle \nabla \mathfrak{g}, \nabla \varphi \rangle + \frac{3}{2 \varphi} \abs{\nabla \varphi}^2 \abs{A}^2. \\
\end{align*}
Using the identities
\begin{align*}
\varphi' = \frac{2 \varphi^2}{v^3}, \quad \varphi'' = \frac{8 \varphi^3 - 6 \varphi^2 v^2}{v^6},
\end{align*}
we obtain 
\begin{align*}
&- \frac{2}{v} \abs{A}^2 \varphi' \abs{\nabla v}^2 + \varepsilon \abs{A}^2 \abs{\nabla v}^2 \varphi' -\abs{A}^2 \varphi'' \abs{\nabla v}^2 + \frac{3}{2 \varphi} \abs{A}^2 \abs{\nabla \varphi}^2 \\
&= - \abs{A}^2 \abs{\nabla \varphi}^2 \left[ \left( \frac{2}{v} - \varepsilon \right) \frac{1}{\varphi'} + \frac{\varphi''}{{\varphi'}^2} - \frac{3}{2 \varphi} \right] \\
&= -\abs{A}^2 \abs{\nabla \varphi}^2 \left[ \left(\frac{2}{v} - \epsilon \right) \frac{v^3}{2 \varphi^2} + \frac{8 \varphi^3 - 6 \varphi^2 v^2}{v^6} \frac{v^6}{4 \varphi^4} - \frac{3}{2 \varphi} \right] \\
&= -\abs{A}^2 \abs{\nabla \varphi}^2 \left[ \frac{1}{2 \varphi} + \frac{v^2}{2 \varphi^2} ( 1 + \epsilon v ) \right] \\
&\leq 0.
\end{align*}
We also have
\begin{align*}
- \abs{A}^4 \varphi' v + 2 \abs{A}^4 \varphi 
&= -2 \varphi^2 \abs{A}^4 \left( \frac{1}{v^2} - \frac{1}{\varphi} \right) \\
&= -2 k \mathfrak{g}^2,
\end{align*}
and
\begin{align*}
\frac{2 r'}{r} H \abs{A}^2 v^2 \varphi' 
&\leq \frac{2 r'}{r} \sqrt{n} \abs{A}^3 v^2 \frac{2 \varphi^2}{v^2} \\
&=4 \sqrt{n} \frac{r'}{r} \frac{\sqrt{\varphi}}{v} \mathfrak{g}^{3/2}.
\end{align*}
We define $\varepsilon = 1/ n v$, then we have
\begin{align*}
&\frac{1}{\varepsilon} \frac{{r'}^2}{r^2} \abs{A}^2 \varphi' -n \frac{{r'}^2}{r^2} v \abs{A}^2 \varphi' 
- \frac{1}{r^2} \{ n (r r'' - {r'}^2) + \mathrm{Ric}_M(v_M, v_M) \} \left( v - \frac{1}{v} \right) \abs{A}^2 \varphi' \\
&+ 2 \varphi \abs{A}^2 \overline{\mathrm{Ric}}(N,N)
+8 \varphi \abs{A}^2 \abs{\overline{R}}_{\overline{M}} + 4 \varphi \abs{A} \abs{\overline{\nabla} \overline{R}}_{\overline{M}}  \\
&=  \left[ -\frac{2}{r^2} \frac{\varphi}{v^3} \{ n (r r'' - {r'}^2) + \mathrm{Ric}_M(v_M, v_M) \} \left( v - \frac{1}{v} \right) 
+2 \overline{\mathrm{Ric}}(N,N)
+8 \abs{\overline{R}}_{\overline{M}} \right] \mathfrak{g} 
+ 4 \sqrt{\varphi} \abs{\overline{\nabla} \overline{R}}_{\overline{M}} \sqrt{\mathfrak{g}}.
\end{align*}
Thus we obtain
\begin{align*}
(\partial_t - \Delta) \mathfrak{g} 
&\leq 
-2 k \mathfrak{g}^2
+4 \sqrt{n} \frac{r'}{r} \frac{\sqrt{\varphi}}{v} \mathfrak{g}^{3/2} 
+ 4 \sqrt{\varphi} \abs{\overline{\nabla} \overline{R}}_{\overline{M}} \sqrt{\mathfrak{g}}  \\
& + \left[ -\frac{2}{r^2} \frac{\varphi}{v^3} \{ n (r r'' - {r'}^2) + \mathrm{Ric}_M(v_M, v_M) \} \left( v - \frac{1}{v} \right) 
+2 \overline{\mathrm{Ric}}(N,N)
+8 \abs{\overline{R}}_{\overline{M}} \right] \mathfrak{g}.
\end{align*}

\end{proof}

Now, let us define $F \colon M \times [0,T) \to \overline{M}$ by $F(p,t) = (p, z(t))$. 
The mean curvature of this hypersurface is $n r'(z(t)) / r(z(t))$, and equation (\ref{eq:MCF}) is equivalent to the following equation 
\begin{align}
\label{eq:standardn}
    \frac{d z}{d t}(t) = -n \frac{r'(z(t))}{r(z(t))}.
\end{align}
Therefore, the flow remains parallel. This solution is used to control the motion of the mean curvature flow. 
\begin{lemma}
Assume the conditions as in Lemma \ref{lem:curv-dist} and let $z_1$, $z_2$ be solutions to equation (\ref{eq:standardn}) with the initial condition $z_2(0) < z_1(0)$. 
Then we have the following:
\begin{align*}
0 < z_1(t) - z_2(t) \leq ( z_1(0) - z_2(0) ) \left( \frac{r(z_2(t))}{r(z_2(0))} \right)^{\alpha},
\end{align*}
for $t > 0$.
\end{lemma}
\begin{proof}
The inequality $z_1(t) - z_2(t) > 0$ follows from the comparison principle for the mean curvature flow. 
Thus, we will show the second inequality and set $w(z) = r'(z)/r(z)$ and $\rho(t) = z_1(t) - z_2(t)$. 
Using Lemma \ref{lem:curv-dist} we have the following:
\begin{align*}
\frac{d \rho}{d t}(t) 
&= -n w(z_1(t)) + n w(z_2(t)) \\
&\leq -n \alpha \rho(t) w(z_1(t)) w(z_2(t)) \\
&\leq -n \alpha \rho(t) w(z_2(t)) w(z_2(t)) \\
&\leq \alpha \rho(t) \frac{r'}{r}(z_2(t)) \frac{d z_2}{d t}(t) \\
&= \alpha \rho(t) \frac{d}{d t} \log r(z_2(t)).
\end{align*}
Thus we have
\begin{align*}
\frac{d}{d t} \log \rho(t) \leq \alpha \frac{d}{d t} \log r(z_2(t)),
\end{align*}
and integrate it to obtain
\begin{align*}
\log \frac{\rho(t)}{\rho(0)} \leq  \alpha \log \frac{r(z_2(t))}{r(z_2(0))}.
\end{align*}
This proves the lemma.
\end{proof}
From this lemma, if $r(z) \to 0$ as $z \to -\infty$, the distance between two distinct parallels goes to zero along the mean curvature flow. 


\section{Motion of curves}
\label{sec:case1}
In this section, we will treat the case $n = 1$. 
We first state the graph-preserving property of the curve shorteing flow.

\begin{theorem}
\label{thm:graph1}
For any warping function $r \colon I \to \mathbb{R}$, the flow $\{ F_t \}_{t \in [0,T}$ remains a graph.
Furthermore, if we assume $r(z) r''(z) - 2 {r'(z)}^2 \geq 0$ for $z \in I$, then $v_t$ is uniformly bounded.
\end{theorem}
\begin{proof}
We define a constant $C(t_1)$ depending on $t_1 < T$ by
\begin{align*}
C(t_1) = \max \{ -\widetilde{C}(t_1), 0 \}, \quad \widetilde{C}(t_1) = \sup_{z \in I(t_1)}\frac{r(z) r''(z) - 2 {r'(z)}^2}{r(z)^2},
\end{align*}
where $I(t_1) = z(\mathbb{S}^1 \times [0,t_1))$, that is, the projection of the image $F(\mathbb{S}^1 \times [0,t_1))$ to $I$ by $\pi_z$. 
We set $v_{\max}(t) = \max_{p \in \mathbb{S}^1} v_t(p)$ and, using Hamilton's trick, we obtain
\begin{align*}
    \frac{d v_{\max}}{d t}(t) \leq  C(t_1) v_{\max}(t),
\end{align*}
from Lemma \ref{lem:v}. Integrate the above to get the following:
\begin{align}
\label{eq:vmax}
v_{\max}(t) \leq v_{\max}(0) e^{C(t_1) t}.
\end{align}
Thus, the flow preserves a graph in $[0,t_1)$ for any choice of $t_1 < T$, and so it does in $[0,T)$. 
If we have $r r'' - 2 {r'}^2 \geq 0$, we can take $C(t_1) = 0$ for any $t_1 < T$. 
Therefore, we obtain $\sup_{(p,t) \in \mathbb{S}^1 \times [0,T)} v_t(p) < \infty$.
\end{proof}

\begin{lemma}
\label{lem:longtime-gen}
Assume the solution to the differential equation $(\ref{eq:standardn})$ for $n = 1$
\begin{align}
\label{eq:standard1}
    \frac{d z}{d t}(t) = - \frac{r'(z(t))}{r(z(t))}, \quad z(0) = z_0,
\end{align}
is defined on $[0,\infty)$ for any initial condition $z(0) = z_0 \in I$.
Then the flow $\{ F_t \}_t$ also exists in $[0,\infty)$.
\end{lemma}
\begin{proof}
We will prove by contradiction; we suppose that $T<\infty$. 
We can take a proper open subinterval $(b, c)$ of $I$ that includes $z_0(\mathbb{S}^1) = \pi_z (F_0(\mathbb{S}^1) )$, the image of the projection of the initial curve. 
The solution to $(\ref{eq:standard1})$ exists in $[0,\infty)$ for any initial condition, in particular for $z(0) = b$ and $z(0)= c$. 
By the comparison principle, the image $F(\mathbb{S}^1 \times [0,T))$ is included in a bounded region, since the flows defined by $(\ref{eq:standard1})$ with the initial conditions $z(0) = b$ and $z(0) = c$ are contained in some bounded (especially compact) region in finite time.  
We then write the region as $B$, that is, $F(\mathbb{S}^1 \times [0,T)) \subset B$. 
The set $\pi_z(B)$ includes $I(t_1)$ defined in the proof of Theorem \ref{thm:graph1} for all $t_1 < T$, therefore we can take the constant $C$ that is independent of $t_1$. 
Therefore, we have
\begin{align*}
v_{\max}(t) \leq v_{\max}(0) e^{C T},
\end{align*}
and this implies that $\sup_{\mathbb{S}^1 \times [0,T)} v$ is bounded, and we can consider the function $\mathfrak{g}$ in Lemma $\ref{lem:g}$ on $\mathbb{S}^1 \times [0,T)$.
Since the flow is contained in a compact set $B$, all the coefficients that appear in Lemma $\ref{lem:g}$ are bounded above, so for some positive constants $D_1$, $D_2$ we have
\begin{align*}
    \left(\partial_t - \Delta\right) \mathfrak{g}
    &\leq -2 k \mathfrak{g}^2 + 4 \frac{r'}{r} \frac{\sqrt{\varphi}}{v^3} \mathfrak{g}^{3/2} - \left\{ \frac{r r'' - 2 {r'}^2}{r^2} \frac{\varphi}{v^3} \left(v - \frac{1}{v} \right) + 2 \frac{r''}{r} \right\} \mathfrak{g} \\
    &\leq -2 k \mathfrak{g}^2 + D_1 \mathfrak{g}^{3/2} + D_2 \mathfrak{g}.
\end{align*}
From the assumption that the maximal existence time $T$ is finite, it is known that the norm of the second fundamental form $\abs{A} ( = \abs{\kappa})$ must be unbounded under the mean curvature flow.
In this case, however, we have the bounded curvature under the flow. 
This is a contradiction; in the remainder, it is sufficient to show that $\mathfrak{g}$ is bounded.
For any $t_0 < T$, we choose $(x_1,t_1) \in \mathbb{S}^1 \times [0,t_0]$ that satisfies $\eta(t_0) = \max_{(x,t) \in \mathbb{S}^1 \times [0,t_0]}\mathfrak{g}(x,t)=\mathfrak{g}(x_1,t_1)$. 
We can assume $t_1 > 0$ and from the above inequality, we can rewrite it as follows:
\begin{align*}
    \mathfrak{g}(x_1,t_1)^2 \leq
    D_1 \mathfrak{g}(x_1,t_1)^{3/2} + D_2 \mathfrak{g}(x_1,t_1),
\end{align*}
that is, 
\begin{align*}
    \eta(t_0)^2 \leq
    D_1 \eta(t_0)^{3/2} + D_2 \eta(t_0).
\end{align*}
This shows that $\eta$ is bounded in $[0,T)$, as is $\mathfrak{g}$. 
\end{proof}

\begin{lemma}
\label{lem:longtime}
We assume that $I =  (-\infty, a)$ and $r'(z) > 0$ in $I$. 
If we have
\begin{align*}
\sup_{z \in I} \frac{r'(z)}{r(z)} < \infty,
\end{align*}
then the solution to $(\ref{eq:standard1})$ exists on $[0,\infty)$ for any $z(0) \in I$.
\end{lemma}
\begin{proof}
Let us set $M = \sup_{z \in I} r'(z) / r(z)$. 
From $(\ref{eq:standard1})$, we have the inequality
\begin{align*}
    \frac{d z}{d t}(t) = - \frac{r'(z(t))}{r(z(t))} \geq -M.
\end{align*}
Integrate it from $0$ to $t$ to obtain
\begin{align*}
z(t) \geq z(0) - M t.
\end{align*}
Suppose that the maximal existence time $T$ is finite; then the flow defined by $(\ref{eq:standard1})$ would converge, and we could extend the solution beyond time $T$. 
This does not occur; thus the lemma is proved.
\end{proof}

\begin{rem}
\label{rem:1}
We remark that if $\sup r' / r = \infty$, Lemma \ref{lem:longtime} does not necessarily hold. 
Let us first consider $r(z) = e^{-z^2}$ for $z \in (-\infty, 0)$, then we have $r'(z) / r(z) = -2z$. 
By a simple calculation, we get the solution to $(\ref{eq:standard1})$,
\begin{align*}
z(t)  =  z(0) e^{2 t}, \quad z(0) < 0.
\end{align*}
Thus, $z(t)$ can be defined in $[0,\infty)$ for any $z(0) \in (-\infty, 0)$.
We then set $r(z) = e^{- e^{-z}}$ for $z \in \mathbb{R}$. We get $r'(z) / r(z) = e^{-z}$ and the solution to $(\ref{eq:standard1})$
\begin{align*}
z(t) = \log(e^{z(0)} - t), \quad z(0) \in \mathbb{R}.
\end{align*}
This solution goes to infinity as $t \to e^{z(0)}$, therefore it cannot extend to infinity beyond $e^{z(0)}$.
\end{rem}

\begin{rem}
\label{rem:2}
Let us consider that $I$ is a bounded open interval and set $I = (0, a)$ with $a > 0$. 
Furthermore, we assume that $r$ is a smooth function of $\mathbb{R}$, $r(0) = 0$, and $r'(z) > 0$ for $z > 0$. 
We cannot define the metric at $z = 0$, so we can only consider it in $I$.
In this case, the maximal existence time $T$ of the flow is always finite. 
This seems to be obvious, but we will give a short proof. 
The ratio $r'(z) / r(z)$ is equal to the derivative of $\log r(z)$, hence we have 
\begin{align*}
\lim_{z \to +0 }\int_{z}^{a} \frac{r'(z)}{r(z)} d z = \log r(a) - \lim_{z \to +0 }\log r(z) = \infty.
\end{align*}
This shows that $r'(0) / r(0) = \infty$. 
Thus, we can take some positive number $\varepsilon > 0$ such that $r'(z) / r(z) > 1$ for all $0 < z < \varepsilon$, and
for the solution to (\ref{eq:standard1}) with $z(0) < a$, we can take $t_{\varepsilon}$ such that $0  < z(t) < \varepsilon$ for all $t > t_{\varepsilon}$. 
Then we have 
\begin{align*}
\frac{d z}{d t}(t) < - 1
\end{align*}
for $t > t_{\varepsilon}$, and integrate it to obtain
\begin{align*}
- \varepsilon < z(t) - z(t_{\varepsilon}) < t_{\varepsilon} - t.
\end{align*}
Therefore, the flow exists only in finite time. 
Although the maximal existence time is finite, the curvature does not necessarily blow up under the flow. 
CSF on surfaces with singularities is studied by Ma \cite{biaoma2022}. Ma showed that CSF on a flat cone (more generally, on a conic Riemann surface) converges to a singularity of the cone under some cone angle conditions.
\end{rem}

\begin{lemma}
\label{lem:derivatives}
If $\abs{r^{(i)}/r} < \infty$ for all $i \geq 1$, we have
\begin{align*}
(\partial_t - \Delta) (\partial_{s}^{(m)} \kappa )^2
\leq -2 (\partial_{s}^{(m+1)} \kappa )^2 &+ 2 \left\{ (m+3) \kappa^2 + \overline{K} \right\} (\partial_{s}^{(m)}\kappa)^2 + 2  P(\abs{\kappa}, \dots, \abs{ \partial_{s}^{(m -1)} \kappa }) \abs{\partial_{s}^{(m)} \kappa}.
\end{align*}
for all $m \geq 1$, where $P$ is a polynomial, which varies with different integers $m$. 
In particular, $P$ does not include any constant term.
\end{lemma}
\begin{proof}
We have the evolution equation for the curvature 
\begin{align*}
(\partial_t - \Delta) \kappa = \kappa^3 + \kappa \overline{K},
\end{align*}
and the time derivative $\partial_t$ and the derivative with respect to arc length $\partial_s$ commute as follows:
\begin{align*}
\partial_t \partial_s = \partial_s \partial_t + \kappa^2 \partial_s.
\end{align*}
Thus, we have 
\begin{align*}
\partial_t \partial_s \kappa
&= \partial_s \partial_t \kappa + \kappa^2 \partial_s \kappa \\
&= \partial_s \Delta \kappa + 3 \kappa^2 \partial_s \kappa + \partial_s \kappa \overline{K} + \kappa \partial_s \overline{K} + \kappa^2 \partial_s \kappa\\
&= \Delta \partial_s \kappa + (4 \kappa^2 + \overline{K}) \partial_s \kappa + \partial_s \overline{K} \kappa.
\end{align*}
Assume that for $\partial_{s}^{(m)} \kappa$ the following holds:
\begin{align}
\label{eq:partial kappa}
(\partial_t - \Delta) \partial_s^{(m)} \kappa
= \left\{ (m+3) \kappa^2 + \overline{K} \right\} \partial_{s}^{(m)}\kappa + \sum_{i=0}^{m-1} P_i (\overline{K}, \dots, \partial_s^{(m)} \overline{K}, \kappa, \dots, \partial_{s}^{(i -1)} \kappa) \partial_s^{(i)} \kappa,
\end{align}
where each $P_i$ is a polynomial.
Then for $m+1$, we obtain
\begin{align*}
(\partial_t - \Delta) \partial_{s}^{(m+1)} \kappa 
&= \partial_s( \partial_t - \Delta) \partial_{s}^{(m)}\kappa + \kappa^2 \partial_{s}^{(m+1)} \kappa \\
&=\partial_s \left\{ (m+3) \kappa^2 + \overline{K} \right\} \partial_{s}^{(m)}\kappa + \left\{ (m+4) \kappa^2 + \overline{K} \right\} \partial_{s}^{(m+1)}\kappa \\
&+ \sum_{i=0}^{m-1} \partial_s  P_i (\overline{K}, \dots, \partial_s^{(m)} \overline{K}, \kappa, \dots, \partial_{s}^{(i -1)} \kappa) \partial_s^{(i)} \kappa \\
&+ \sum_{i=0}^{m-1}  P_i (\overline{K}, \dots, \partial_s^{(m)} \overline{K}, \kappa, \dots, \partial_{s}^{(i -1)} \kappa) \partial_s^{(i+1)} \kappa \\
&= \left\{ (m+4) \kappa^2 + \overline{K} \right\} \partial_{s}^{(m+1)}\kappa 
+ \sum_{i=0}^{m} \widetilde{P}_i (\overline{K}, \dots, \partial_s^{(m+1)} \overline{K}, \kappa, \dots, \partial_{s}^{(i -1) \kappa }) \partial_s^{(i)} \kappa,
\end{align*}
for some polynomials $\widetilde{P}_i$, and thus the identity $(\ref{eq:partial kappa})$ holds for all $m \geq 0$ by induction. 
Also we have
\begin{align*}
\abs{\partial_s \overline{K}} 
&= \abs{- \overline{K}' \langle N, E_{\theta} \rangle} \leq \widehat{B} \\
\abs{\partial_s^2 \overline{K}} 
&= \abs{-\partial_s \overline{K}' \langle N, E_{\theta} \rangle - \overline{K}' \partial_s \langle N, E_{\theta} \rangle} \\
&= \abs{- \overline{K}'' \langle N, E_{\theta} \rangle^2 - \overline{K}' \langle \kappa \mathfrak{t}, E_{\theta} \rangle - \overline{K}' \langle N, -r'/r E_z \rangle \langle \mathfrak{t}, E_{\theta} \rangle} \leq \widehat{C} + \widehat{C}_1 \abs{\kappa},
\end{align*}
for some constants $\widehat{B}$, $\widehat{C}$, and $\widehat{C}_1$. 
Thus we similarly obtain
\begin{align*}
\abs{\partial_s^{(m)} \overline{K}}
\leq C + \sum_{i=0}^{m-2} C_i \abs{\partial_s^{(i)} \kappa}
\end{align*}
for some constants $C$ and $C_i$. 
Hence we can obtain
\begin{align*}
(\partial_t - \Delta) (\partial_s^{(m)} \kappa)^2 
&= 2 \partial_s^{(m)} \kappa (\partial_t - \Delta) \partial_s^{(m)} \kappa - 2 (\partial_s^{(m+1)} \kappa)^2 \\
&=- 2 (\partial_s^{(m+1)} \kappa)^2 + 2 \left\{ (m+3) \kappa^2 + \overline{K} \right\} (\partial_{s}^{(m)}\kappa)^2 \\
&+2 \sum_{i=0}^{m-1} P_i (\overline{K}, \dots, \partial_s^{(m)} \overline{K}, \kappa, \dots, \partial_{s}^{(i -1)} \kappa) (\partial_s^{(i)} \kappa)(\partial_s^{(m)} \kappa) \\
&\leq - 2 (\partial_s^{(m+1)} \kappa)^2 + 2 \left\{ (m+3) \kappa^2 + \overline{K} \right\} (\partial_{s}^{(m)}\kappa)^2 + P(\abs{\kappa}, \dots, \abs{ \partial_{s}^{(m -1)} \kappa }) \abs{\partial_s^{(m)} \kappa}
\end{align*}
for some polynomial $P$.
\end{proof}

Using this inequality, we obtain uniform bounds of the derivatives of the curvature if we have a uniform bound of the curvature.

\begin{lemma}
\label{lem:curvature-bounded}
If $\sup_{\mathbb{S}^1 \times [0,T)}\kappa^2 < \infty$ and $\abs{r^{(i)}/r} < \infty$ for all $i \geq 1$, we have $\sup_{\mathbb{S}^1 \times [0,T)} (\partial_{s}^{(m)} \kappa )^2 < \infty$ for all $m \geq 1$. 
\end{lemma}
\begin{proof}
We prove this by induction. 
Assume that $(\partial_{s}^{(k)} \kappa )^2$ are uniformly bounded for $k=1, \dots, m-1$, then we have
\begin{align*}
&(\partial_t - \Delta) ((\partial_{s}^{(m)} \kappa )^2 + B (\partial_{s}^{(m-1)} \kappa )^2) \\
&\leq 2 \left\{ -B + (m+3) \kappa^2 + \overline{K} \right\} (\partial_{s}^{(m)}\kappa)^2 
+ P(\abs{\kappa}, \dots, \abs{ \partial_{s}^{(m -1)} \kappa }) \abs{\partial_{s}^{(m)} \kappa} \\
&+2 B\left\{ (m+2) \kappa^2 + \overline{K} \right\} (\partial_{s}^{(m-1)}\kappa)^2 
+ B \widehat{P}(\abs{\kappa}, \dots, \abs{ \partial_{s}^{(m -2)} \kappa }) \abs{\partial_{s}^{(m-1)} \kappa}
\end{align*}
 for some constant $B > 0$. We choose a large $B$ and use $ \abs{\partial_{s}^{(m)} \kappa} \leq  (\partial_{s}^{(m)} \kappa)^2 + 1$, then we have
\begin{align}
\label{eq:derivatives}
\begin{split}
    (\partial_t - \Delta) ((\partial_{s}^{(m)} \kappa )^2 + B (\partial_{s}^{(m-1)} \kappa )^2) 
\leq& - \widetilde{B}  ((\partial_{s}^{(m)} \kappa )^2 + B (\partial_{s}^{(m-1)} \kappa )^2) \\
&+ \widetilde{P}(\abs{\kappa}, \dots, \abs{ \partial_{s}^{(m -1)} \kappa }),
\end{split}
\end{align}
for some constant $\widetilde{B} > 0$ and polynomial $\widetilde{P}$. 
This inequality implies that $ (\partial_{s}^{(m)} \kappa)^2$ is uniformly bounded because $(\partial_{s}^{(k)} \kappa )^2$ is uniformly bounded for $k= 0,1, \dots, m-1$. 
By induction, all derivatives are bounded; hence, the lemma is proved.
\end{proof}

Finally, we will prove Theorem $\ref{thm:1}$.

\begin{proof}[Proof of Theorem \ref{thm:1}]

Theorem \ref{thm:graph1}, Lemma \ref{lem:longtime-gen} and Lemma \ref{lem:longtime} prove (i) and (ii-a). 
We will prove the first statement of (ii-b). 
Set $\mathfrak{g}_{\max}(t) = \max_{p \in \mathbb{S}^1} \mathfrak{g}_t(p)$. 
From Lemma \ref{lem:g} and Hamilton's trick, we have
\begin{align*}
\frac{d \mathfrak{g}_{\max}}{d t}
&\leq \left\{ - 2 k\mathfrak{g}_{\max} + \frac{4 r'}{r} \frac{\sqrt{\varphi}}{v^3} {\sqrt{\mathfrak{g}_{\max}}}
- \frac{r r'' - 2 {r'}^2 }{r^2} \left( v - \frac{1}{v} \right) \frac{2 \varphi}{v^3}  - 2 \frac{r''}{r} \right\} \mathfrak{g}_{\max} \\
&\leq \left( - 2 k\mathfrak{g}_{\max} + \frac{4 r'}{r} \frac{\sqrt{\varphi}}{v^3} {\sqrt{\mathfrak{g}_{\max}}} \right)\mathfrak{g}_{\max} \\
&\leq \left( - 2 k\mathfrak{g}_{\max} + C \frac{ r'}{r} \right) \mathfrak{g}_{\max},
\end{align*}
for a positive constant $C > 0$. 
For any $\varepsilon > 0$, we can take a large $t_{\varepsilon}$ so that for $t > t_{\varepsilon}$
\begin{align*}
C\frac{r'}{r}(z(t)) < 2 k \varepsilon,
\end{align*}
from Lemma \ref{lem:curv-dist} since the flow goes to negative infinity by the comparison principle. 
Thus, we have
\begin{align*}
\frac{d \mathfrak{g}_{\max}}{d t}(t) \leq 2 k (-\mathfrak{g}_{\max}(t) + \varepsilon) \mathfrak{g}_{\max}.
\end{align*}
If $\mathfrak{g}_{\max}(t') > \varepsilon$ for some $t' \geq t_{\varepsilon}$, we have $(d \mathfrak{g}_{\max}/d t)(t') < 0$, thus we can consider the following two cases:
\begin{itemize}
\item There exists a time $t'_{\varepsilon} \geq t_{\varepsilon}$ such that $\mathfrak{g}_{\max}(t) \leq \varepsilon$ for all $t \geq t'_{\varepsilon}$; or 
\item  $\mathfrak{g}_{\max}(t) > \varepsilon$ and $(d \mathfrak{g}_{\max} / d t) (t) < 0$ for all $t > t_{\varepsilon}$.
\end{itemize}
However, the latter case does not occur. In fact, we suppose that the latter holds. 
Then there is a limit $\lim_{t \to \infty} \mathfrak{g}_{\max}(t) = a \geq \varepsilon$. We can take a large $t_0$ such that
\begin{align*}
\frac{d \mathfrak{g}_{\max}}{d t}
&\leq \left\{ - 2 k\mathfrak{g}_{\max}(t) + \frac{4 r'}{r} \frac{\sqrt{\varphi}}{v^3} {\sqrt{\mathfrak{g}_{\max}}} \right\}\mathfrak{g}_{\max} \\
&\leq -D \mathfrak{g}_{\max}(t)
\end{align*}
for some positive constant $D$ and $t \geq t_0$ since $r'(z) / r(z)$ goes to zero from Lemma \ref{lem:curv-dist} in Section \ref{sec:basic}. Thus, we have
\begin{align*}
\mathfrak{g}_{\max}(t) \leq \mathfrak{g}_{\max}(t_0) e^{- D t} \to 0.
\end{align*}
This contradicts the assumption; for any $\varepsilon > 0$, we have $\mathfrak{g}_{\max}(t) \leq \varepsilon$ for $t \geq t'_{\varepsilon}$, and this implies that $\kappa_t \to 0$ as $t \to \infty$. 
We will show the second statement of (ii-b), so in addition, we assume that $\sup_{z \in I}\abs{r^{(i)}(z)/r(z)} < \infty$ for all $i \geq 2$. 
Then, by Lemma \ref{lem:curvature-bounded}, $(\partial_{s}^{(m)} \kappa )^2$ is uniformly bounded for all $m \geq 1$. 
Since $\kappa \to 0$, we can prove $\partial_{s}^{(m)} \kappa \to 0$ by induction. We assume that $\partial_{s}^{(k)} \kappa$ converges to zero for $k = 0, 1, \dots, m -1$.
We set the following function
\begin{align*}
\rho(t) = \max_{p \in \mathbb{S}^1}  \left\{(\partial_{s}^{(m)} \kappa )^2 + B (\partial_{s}^{(m-1)} \kappa )^2\right\}(p, t),
\end{align*}
where $B$ is defined as in the proof of Lemma \ref{lem:curvature-bounded}. 
From inequality (\ref{eq:derivatives}) we have
\begin{align*}
\frac{d \rho}{d t}(t) 
&\leq - \widetilde{B} \rho(t) + \widetilde{P}(\abs{\kappa}, \dots, \abs{ \partial_{s}^{(m -1)} \kappa })\\
&\leq - \widetilde{B} \rho(t) +\delta(\tau)
\end{align*}
for $t \in (\tau, \infty)$, where $\delta(\tau) = \sup_{\mathbb{S}^1 \times (\tau,\infty)} \widetilde{P}(\abs{\kappa}, \dots, \abs{ \partial_{s}^{(m -1)} \kappa }) \to 0$ as $\tau \to \infty$ by the assumption of induction. Then we obtain
\begin{align*}
\rho(t) \leq \frac{\delta(\tau)}{\widetilde{B}} + e^{\widetilde{B}(\tau - t)} \left( \rho(\tau) - \frac{\delta(\tau)}{\widetilde{B}} \right).
\end{align*}
For any $\varepsilon > 0$, we can take sufficiently large $\tau$ such that
\begin{align*}
\rho(t) \leq e^{\widetilde{B}(\tau - t)} \rho(\tau) + \varepsilon,
\end{align*}
then we have
\begin{align*}
    \limsup_{t \to \infty} \rho(t) \leq \varepsilon.
    \end{align*}
This implies $\lim_{t \to \infty} (\partial_{s}^{(m)} \kappa )^2 = 0$ and proves the theorem by induction.
\end{proof}

\section{Motion of hypersurfaces}
\label{sec:case2}
We then consider the case that $n \geq 2$ in this section. 

First we recall the conditions that we assume on the Riemannian manifold $(M, g_M)$ and the function $r \colon (-\infty, a) \to \mathbb{R}$:
\begin{itemize}
\item[(C0)] $\mathrm{Ric}_M \geq n \rho g_M$; 
\item[(C1)] $r'(z) > 0$ for all  $s \in (-\infty,a)$; and
\item[(C2)] $r(z) r''(z) - (1 + \alpha){r'(z)}^2 + \rho \geq c$,
\end{itemize}
where $\rho \in \mathbb{R}$, $c = \max\{\rho, 0\}$ and $\alpha > 1$.
We assume these conditions throughout this section. 
The idea of proofs of Lemma $\ref{lem:f}$ and Theorem \ref{thm:graph2} is originally from the paper \cite{huang2019}.

\begin{lemma}
\label{lem:f}
Define $f = \Theta^2$, then we have 
\begin{align*}
    \left(\partial_t - \Delta\right) f \geq \left\langle \nabla f, \frac{2 r'}{r} E_z - \frac{\nabla f}{2 f} \right\rangle + \frac{2 n (1- f)}{r^2} \{ {r'}^2 (\alpha f -1) + c f \}.
\end{align*}
\end{lemma}
\begin{proof}
From Lemma \ref{lem:theta} and $H^2 \leq n \abs{A}^2$, we can compute it as follows:
\begin{align*}
\left(\partial_t - \Delta\right) f 
&= 2 \Theta (\partial_t - \Delta) \Theta - \frac{\abs{\nabla f}^2}{2 f} \\
&\geq \left\langle \nabla f, 2 \frac{r'}{r}  E_z - \frac{\nabla f}{2 f} \right\rangle - \frac{4 r'}{r} \sqrt{n} \abs{A} \sqrt{f} + 2 \abs{A}^2 f \\
    & \quad + 2 n \frac{{r'}^2}{r^2} f + \frac{2 f (1 - f)}{r^2} \{ n (r r'' - {r'}^2) + \mathrm{Ric}_M(v_M, v_M) \}   \\
&= \left\langle \nabla f, 2 \frac{r'}{r}  E_z - \frac{\nabla f}{2 f}  \right\rangle +2 \left( \abs{A} \sqrt{f} -\frac{\sqrt{n} r'}{r} \right)^2 - 2n \frac{{r'}^2}{r^2} \\
& \quad + 2 n \frac{{r'}^2}{r^2} f + \frac{2 f (1 - f)}{r^2} \{ n (r r'' - {r'}^2) + \mathrm{Ric}_M(v_M, v_M) \}  ]\\
&\geq \left\langle \nabla f, 2 \frac{r'}{r}  E_z - \frac{\nabla f}{2 f}  \right\rangle 
-2 n \frac{{r'}^2}{r^2} (1 - f) + \frac{2 f (1 - f)}{r^2} \{ n (r r'' - {r'}^2) + n \rho  \}\\
&\geq \left\langle \nabla f, 2 \frac{r'}{r}  E_z - \frac{\nabla f}{2 f} \right \rangle 
-2 n \frac{{r'}^2}{r^2} (1 - f) + \frac{2n f (1 - f)}{r^2} (\alpha {r'}^2 + c)  \\
&= \left\langle \nabla f, 2 \frac{r'}{r}  E_z - \frac{\nabla f}{2 f} \right\rangle 
+\frac{2n(1 - f)}{r^2} \{ {r'}^2 (\alpha f - 1) + c f) \}.
\end{align*}

\end{proof}

\begin{theorem}
\label{thm:graph2}
Assume (C0), (C1), (C2) and the initial condition 
  \begin{align*}
  \min_{p \in M_0} \Theta_0(p) > \alpha^{-1/2},
  \end{align*}
hold.
Then, for some positive constant $c$, we have $\Theta_t(x) \geq c$, that is, $F_t$ is a graph for all $t \in [0,T)$.
\end{theorem}

\begin{proof}
From the assumption for the theorem, we have $\min_{p \in M} f_0(p) > 1 / \alpha$. 
Thus, $f_t$ is away from zero in a short time. 
We then set $\mu(t) = min_{p \in M} f_t(p)$, and from Lemma \ref{lem:f} we have 
\begin{align*}
\frac{d \mu}{d t} \geq \frac{2n(1 - \mu)}{r^2} \{ {r'}^2 (\alpha \mu - 1) + c \mu) \} > 0.
\end{align*}
Since $\mu$ increases monotonically, we have $\mu(t) \geq \mu(0) > 0$ for all $t \in [0, T)$. 
This proves that the flow remains a graph.
\end{proof}

\begin{rem}
The estimate of the function $f$ (or $\mu$) is easier to derive than the estimate in \cite{huang2019} since condition (C2) is stronger to treat the case $\lim_{z \to -\infty} r(z) = 0$. 
In Section 3 of \cite{huang2019}, the estimate of the function $f$ is derived with more care using $r(0) = 1 (> 0)$.
\end{rem}

Before we prove Theorem \ref{thm:2}, we give here two lemmas corresponding to Lemma \ref{eq:derivatives} and inequality $(\ref{eq:derivatives})$ (see also the proof of Lemma 7.2 in \cite{huisken1986contracting}). 

\begin{lemma}
\label{lem:higher-A}
For $\abs{\nabla^m A}^2$, we have
\begin{align*}
(\partial_t - \Delta) \abs{\nabla^m A}^2 
\leq &- 2 \abs{\nabla^{m+1} A}^2 + P(\abs{A}, \dots, \abs{\nabla^{m-1} A}, \abs{R}, \dots, \abs{\overline{\nabla}^{m+1} \overline{R}}) \abs{\nabla^m A}^2 \\
&+ Q(\abs{A}, \dots, \abs{\nabla^{m-1} A}, \abs{R}, \dots, \abs{\overline{\nabla}^{m+1} \overline{R}}),
\end{align*}
where $P$ and $Q$ are polynomials and do not contain any constant term.
\end{lemma}

\begin{lemma}
    \label{lem:higher-bound}
    We have
\begin{align*}
(\partial_t - \Delta) ( \abs{\nabla^m A}^2 + B \abs{\nabla^{m-1} A}^2 )
\leq &- \widetilde{B} ( \abs{\nabla^m A}^2 + B \abs{\nabla^{m-1} A}^2 ) \\
&+ \widetilde{P}(\abs{A}, \dots, \abs{\nabla^{m-1} A}, \abs{R}, \dots, \abs{\overline{\nabla}^{m+1} \overline{R}}),
\end{align*}
for some positive constants $B$, $\widetilde{B}$ and polynomial $\widetilde{P}$ that does not have any constant term.
\end{lemma}

\begin{proof}[Proof of Theorem \ref{thm:2}]
The long-time existence of the flow can be proved in the same way as in Lemma \ref{lem:longtime-gen}.
We will then prove (ii). Since $\abs{A}^2$ is uniformly bounded, $\abs{\nabla^i A}^2$ is uniformly bounded for all $i \geq 1$, and this is implied by the same argument as in the proof of Lemma \ref{lem:curvature-bounded} using Lemma \ref{lem:higher-bound}. 
From Lemma \ref{lem:g}, we can prove $\mathfrak{g} \to 0$ as in the proof of Theorem \ref{thm:1}. 
We note the differences between the proofs.
We have the inequality
\begin{align*}
    \frac{d \mathfrak{g}_{\max}}{d t}
    &\leq \biggl\{ - 2 k \mathfrak{g}_{\max}^{3/2} + 4 \sqrt{n} \frac{r'}{r} \frac{\sqrt{\varphi}}{v} \mathfrak{g}_{\max} + 4 \sqrt{\varphi} \abs{\overline{\nabla} \overline{R}}  \\
   &\qquad \qquad + \left[ - \frac{2}{r^2} \frac{\varphi}{v^3}  n (r r'' - {r'}^2) \left( v - \frac{1}{v}\right)  + 2 \overline{\mathrm{Ric}}(N,N) + 8 \abs{\overline{R}}  \right] \sqrt{\mathfrak{g}_{\max}} \biggr\}\sqrt{\mathfrak{g}_{\max}} \\
   &\leq \left\{ - 2 k \mathfrak{g}_{\max}^{3/2}  + P \left( \frac{r'}{r}, \frac{r''}{r}, \frac{r'''}{r} \right)  \right\}\sqrt{\mathfrak{g}_{\max}},
\end{align*}
where $P(r'/r, r''/r, r'''/r) \to 0$ as $z \to -\infty$.
Thus, for any $\varepsilon > 0$, we can take a large $t_{\varepsilon}$ such that 
\begin{align*}
P\left( \frac{r'}{r}, \frac{r''}{r}, \frac{r'''}{r} \right) < 2 k \varepsilon.
\end{align*}
Then we have the following.
\begin{align*}
 \frac{d \mathfrak{g}_{\max}}{d t} \leq 2 k ( - \mathfrak{g}_{\max}^{3/2}  +  \varepsilon ) \sqrt{\mathfrak{g}_{\max}},
\end{align*}
for $t > t_{\varepsilon}$.
If $\mathfrak{g}_{\max}^{3/2}(t') > \varepsilon$ for some $t' \geq t_{\varepsilon}$, we have $(d \mathfrak{g}_{\max}/d t)(t') < 0$, thus we can consider the following two cases:
\begin{itemize}
\item There exists a time $t'_{\varepsilon} \geq t_{\varepsilon}$ such that $\mathfrak{g}_{\max}^{3/2}(t) \leq \varepsilon$ for all $t \geq t'_{\varepsilon}$; or 
\item  $\mathfrak{g}_{\max}^{3/2}(t) > \varepsilon$ and $(d \mathfrak{g}_{\max} / d t) (t) < 0$ for all $t > t_{\varepsilon}$.
\end{itemize}
However, the latter case does not occur. In fact, we suppose that the latter holds. 
Then there is a limit $\lim_{t \to \infty} \mathfrak{g}_{\max}(t) = a \geq \varepsilon$. We can take a large $t_0$ such that
\begin{align*}
\frac{d \mathfrak{g}_{\max}}{d t}(t)
 &\leq \left\{ - 2 k \mathfrak{g}_{\max}^{3/2}  + P \left( \frac{r'}{r}, \frac{r''}{r}, \frac{r'''}{r} \right)  \right\}\sqrt{\mathfrak{g}_{\max}} \\
&\leq -D \sqrt{\mathfrak{g}_{\max}}(t)
\end{align*}
for some positive constant $D$ and $t \geq t_0$ since  $P(r'/r, r''/r, r'''/r) \to 0$ as $z \to -\infty$. Thus we obtain 
\begin{align*}
\sqrt{\mathfrak{g}_{\max}}(t) -\sqrt{\mathfrak{g}_{\max}}(t_0) \leq -\frac{D}{2} (t - t_0).
\end{align*}
However, this is a contradiction. 
Therefore, we have $\abs{A}^2 \to 0$. 
From Lemma \ref{lem:rR} and the assumption of (ii), the polynomial $\widetilde{P}$ in Lemma \ref{lem:higher-bound} goes to zero as $t \to \infty$ if $\abs{\nabla^k A} \to 0$ for $k = 0, 1, \dots, m-1$. 
Thus, also using the same induction argument, we can prove $\abs{\nabla^m A} \to 0$ for $m \geq 1$.
\end{proof}

\noindent
\textit{Acknowledgement.} 
I would like to express my gratitude to my supervisor Naoyuki Koike for valuable advice and encouragement.

\appendix
\def\thesection{Appendix}
\section{A counterexample}

We give an example of an MCF that does not preserve a graph in a warped product manifold. 
The idea of construction comes from \cite{Zhou2017}.
Let $(\overline{M}, \overline{g})$ be an $(n+1)$-dimensional Euclidean space except its origin, that is, we take $(M, g)$ as a standard $n$-dimensional sphere of radius 1, $(\mathbb{S}^n, g_S)$, and define
\begin{align*}
\begin{cases}
\overline{M} = \mathbb{S}^n \times (0,\infty), \\
\overline{g} = z^2 g_S + d z \otimes d z.
\end{cases}
\end{align*}
Let $z_{\varepsilon, r_0, r_1} \colon \mathbb{S}^n \to (0,\infty)$ be define
\begin{align*}
z_{\varepsilon, r_0, r_1}(\omega) = 
\begin{cases}
r_0 & \omega_1 \leq 0 \\
\eta(\sqrt{1 - (\omega_1)^2}) & \omega_1 > 0,
\end{cases}
\end{align*}
where we write $\omega = (\omega_1, \dots, \omega_{n+1}) \in \mathbb{S}^{n} \subset \mathbb{R}^{n+1}$, and we choose a function $\eta \colon [0,1) \to (0,\infty)$ satisfying that
\begin{align*}
\begin{cases}
\eta(y) \equiv r_1 &( 0 \leq y \leq \varepsilon) \\
\eta(y) \equiv r_0 & (2\varepsilon \leq y < 1) \\
d\eta/dy(y) \leq 0 & (0 \leq y < 1)
\end{cases}
\end{align*}
for any positive constant $\varepsilon < 1/2$ and $r_0 < r_1$ (see Figure \ref{fig:eta}). 

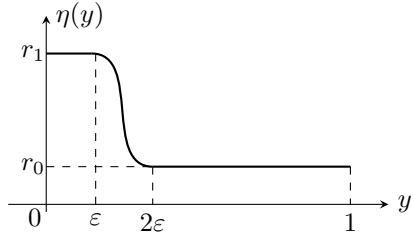
\begin{figure}[htbp]
\centering
\begin{tikzpicture}[domain=-5:5,samples=100,>=stealth]
\draw[->] (-0.5,0) -- (4.5,0) node[right] {$y$};
\draw[->] (0,-0.1) -- (0,2.5) node[right] {$\eta(y)$};

\draw[thick] (0, 2) -- (0.6,2);
\draw[thick] (0.6,2) to [out=0,in=180] (1.4,0.5);
\draw[thick] (1.4,0.5) -- (4,0.5);
\draw[dashed] (0,0.5) -- (1.4,0.5);
\draw[dashed] (0.65,2) -- (0.65,0);
\draw[dashed] (1.4,0.5) -- (1.4,0);
\draw[dashed] (4,0.5) -- (4,0);

\draw (-0.17,2) node {$r_1$};
\draw (-0.17,0.5) node {$r_0$};

\draw (0.65,0) node[below] {$\varepsilon$};
\draw (1.4,0) node[below] {$2\varepsilon$};
\draw (4,0) node[below] {$1$};

\draw (-0.15,-0.18) node{$0$};
\end{tikzpicture}  
    
    \caption{Function $\eta(y)$}
    \label{fig:eta}
\end{figure}

Let $F_0 \colon \mathbb{S}^n \to \overline{M}$ be defined by
\begin{align*}
F_0(\omega) = (\omega, z_{\varepsilon, r_0, r_1}(\omega) ),
\end{align*}
and set $M_0 = F_0(\mathbb{S}^n)$.
The hypersurface $M_0$ is also considered as a surface of revolution around the $x_1$-axis in $\mathbb{R}^{n+1}$ (see Figure \ref{fig:M_0}). 
In fact, for $\omega_1 > 0$, the projection of the hypersurface (depending only on $\omega_1$) onto the $x_1$-axis can be written as 
\begin{align*}
x_1 = 
\begin{cases}
r_0 \omega_1 & (\omega_1 \leq 0) \\
\eta(\sqrt{1 - (\omega_1)^2}) \omega_1 & (\omega_1 > 0)
\end{cases}
\end{align*}
and we have
\begin{align*}
\frac{d x_1}{d \omega_1} &= 
\begin{cases}
r_0 & (\omega_1 \leq 0) \\
\frac{d \eta}{d y}(\sqrt{1 -(\omega_1)^2}) \frac{- (\omega_1)^2}{\sqrt{1 -(\omega_1)^2}} + \eta(\sqrt{1 -(\omega_1)^2}) & (\omega_1 > 0)
\end{cases}\\
&>0.
\end{align*}
Hence $\omega_1$ and $x_1$ have a one-to-one correspondence, and we can define a function $u \colon (-1, r_1) \to (0, \infty)$ on the $x_1$-axis by
\begin{align*}
u(x_1) = 
\begin{cases}
r_0 \sqrt{1 - \omega_1(x_1)^2} & ( -1 < x_1 \leq 0) \\
\eta(\sqrt{1 - \omega_1(x_1)^2}) \sqrt{1 - \omega_1(x_1)^2} & (0 < x_1 < r_1).
\end{cases}
\end{align*}
Rotating a graph of this function $u$ around the $x_1$-axis, we obtain the original hypersurface $M_0$. 

\begin{figure}[htbp]
\centering

\begin{tikzpicture}[domain=-5:5,samples=100,>=stealth]
\draw[->,dashed] (-1.,0) -- (10,0) node[right] {$x_1$};

\draw[thick] (-0.7,0) to [out=90,in=180] (0,0.7);
\draw[thick] (0,-0.7) to [out=180,in=270] (-0.7,0);

\draw[thick] (0,0.7) to [out=0,in=95] (0.65,0.1);
\draw[thick] (0,-0.7) to [out=0,in=265] (0.65,-0.1);

\draw[thick] (0.65,0.1) to [out=265,in=15] (0.75,0.1);
\draw[thick] (0.65,-0.1) to [out=95,in=355] (0.75,-0.1);

\draw[thick] (0.75,0.1) -- (8.5,0.7);
\draw[thick] (0.75,-0.1) -- (8.5,-0.7);

\draw[thick] (8.5,0.7) to [out=0,in=90] (9.2,0);
\draw[thick] (8.5,-0.7) to [out=0,in=270] (9.2,0);

\draw (-1.05,0.12) node{$-r_0$};
\draw (9.4,0.12)  node{$r_1$};

\draw[thick] (0, 0.05) -- (0,-0.05);
\draw (0,0) node[below] {$0$};

\draw (0,-0.7) to [out=180,in=180] (0,0.7);
\draw[dotted] (0,-0.7) to [out=0,in=0] (0,0.7);

\draw (4.5,-0.4) to [out=180,in=180] (4.5,0.4);
\draw[dotted] (4.5,-0.4) to [out=0,in=0] (4.5,0.4);

\draw (8.5,-0.7) to [out=180,in=180] (8.5,0.7);
\draw[dotted] (8.5,-0.7) to [out=0,in=0] (8.5,0.7);

\end{tikzpicture}  
    \caption{Surface of revolution generated by $u$}
    \label{fig:M_0}
\end{figure}
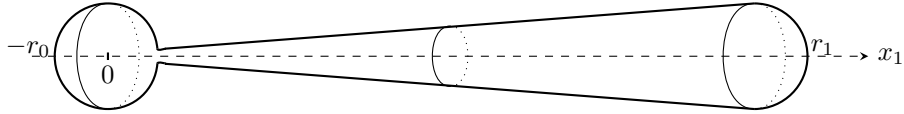
Now, we set the appropriate constants $\varepsilon$, $r_0$, and $r_1$.
First, we define $r_0$ so that $M_0$ contains the sphere shrinking to the origin under MCF at time 2. 
Then, we consider the "shrinking doughnut" of Angenent \cite{Ang} that shrinks to a point on the $x_1$-axis before time 1. 
To make $M_0$ and the shrinking doughnut disjoint, we take $\varepsilon$ sufficiently small and place the shrinking doughnut outside $M_0$.
Finally, $r_1$ is taken large enough so that $M_0$ contains another sphere shrinking to a point on the $x_1$-axis at time 2. 
Finally, we can apply Theorem 1.2 in \cite{AAG1995} to the MCF with the initial hypersurface $M_0$, then the neck will be pinched before time 1, and the hypersurface just before the singular time is no longer a geodesic graph.


\vspace{0.5truecm}

\begin{flushright}
{\small 
Naotoshi Fujihara \\
Department of Mathematics \\
Graduate School of Science \\
Tokyo University of Science \\
1-3 Kagurazaka \\
Shinjuku-ku \\
Tokyo 162-8601 \\
Japan \\
(E-mail: 1123706@ed.tus.ac.jp)
}
\end{flushright}

\end{document}